\documentclass[a4paper,12pt]{article}     % `article' for A4 paper
\usepackage{amsmath,amscd,amssymb}        % AMS-LaTeX with CD and symbols
\usepackage{latexsym}                     % old LaTeX symbols
\usepackage[english, german]{babel}                % German & English (default)

\usepackage{theorem}                 % theorem extensions

\newtheorem{theorem}{Theorem}

\newenvironment{example}
{\smallskip\noindent{\bf Example\/}.}{\smallskip\par}

\newenvironment{remarks}
{\smallskip\noindent{\bf Remarks\/}.}{\smallskip\par}
\newenvironment{proof}
{\noindent{\it Proof\/}.}{{ \hfill $\Box$}\smallskip\par}

\newcommand{\CC}{{\mathbb C}}

\newcommand{\ZZ}{{\mathbb Z}}

\newcommand{\ww}{\underline{w}}

\newcommand{\rk}{{\rm rk}\,}

\title{Monodromy of dual invertible polynomials}
\author{W.~Ebeling and S.~M.~Gusein-Zade
\thanks{Partially supported by the DFG (Eb 102/7--1), RFBR--10-01-00678, NSh--8462.2010.1.
}
}
\date{}

\begin{document}
\selectlanguage{english}

\maketitle

\begin{abstract}
A generalization of Arnold's strange duality to invertible polynomials in three variables by the first author and A.~Takahashi includes the following relation. For some invertible polynomials $f$ the Saito dual of the reduced monodromy zeta function of $f$ coincides with a formal ``root'' of the reduced monodromy zeta function of its Berglund-H\"ubsch transpose $f^T$. Here we give a geometric interpretation of  ``roots'' of the monodromy zeta function and generalize the above relation to all non-degenerate invertible polynomials in three variables and to some polynomials in an arbitrary number of variables in a form including  ``roots'' of the monodromy zeta functions both of $f$ and $f^T$.
\end{abstract}

\section*{Introduction} 
V.~Arnold observed a strange duality between the 14 exceptional unimodular singularities \cite{Ar}. One of the features of this duality is that dual singularities have Saito dual characteristic polynomials of the classical monodromy transformations \cite{S1, S2}.  For a fixed positive integer $d$ and for a rational function $\varphi(t)$ of the form $\prod\limits_{m|d}(1-t^m)^{s_m}$, the Saito dual %% function
of $\varphi$ with respect to $d$ is
$$
\varphi^\ast(t) = \prod\limits_{m|d}(1-t^{d/m})^{-s_m}.
$$
(For quasihomogeneous singularities, in particular for the quasihomogeneous representatives of the 14 exceptional unimodular singularities, $d$ is the (quasi)de\-gree of the function.)

A system of weights $\ww=(w_1, \ldots, w_n;d)$ for a quasihomogeneous  polynomial $f$  defines a natural grading on the ring $\CC[x_1, \ldots, x_n]/(f)$ with the Poincar\'e series
$$
P_{\ww}(t)=\frac{1-t^{d}}{\prod\limits_{j=1}^n (1-t^{w_j})}\,.
$$
Here, in general, we do dot assume that the system of weights is reduced, i.e., we permit the greatest common divisor of the weights $w_1, \ldots , w_n$ to be greater than 1.
%% ($d'$ is the (quasi)degree of the polynomial  $f$).
The system of weights $\ww$ defines a $\CC^\ast$-action on $\CC^n$:
$$
\lambda\ast(x_1, \ldots, x_n)=(\lambda^{w_1}x_1, \ldots, \lambda^{w_n}x_n).
$$
This action may have orbits of different types (i.e., with different isotropy subgroups) on the zero-level set $X_f$ of the polynomial $f$. The orbit invariants give rise to a rational function ${\rm Or}_{\ww}(t)$: see, e.g., \cite{Trieste}.

Let $\widetilde{\zeta}_f(t)=\zeta_f(t)/(1-t)$ be the reduced zeta function of the classical monodromy transformation $h_f:V_f\to V_f$ of $f$ ($V_f$ is the Milnor fibre of $f$). For $n=3$ and for an isolated singularity $f$, the reduced zeta function $\widetilde{\zeta}_f(t)$ coincides with the characteristic polynomial of the monodromy transformation. In \cite{MRL} it was shown that, for the reduced system of weights $\ww$ of $f$, one has
\begin{equation}\label{ourSaito}
P_{\ww}(t){\rm Or}_{\ww}(t)=\widetilde{\zeta}^\ast_f(t),
\end{equation}
where $\widetilde{\zeta}^\ast_f(t)$ is the Saito dual (with respect to the degree $d$) of the reduced zeta function $\widetilde{\zeta}_f(t)$.

Arnold's strange duality can be considered as a special case of mirror symmetry. In \cite{ET} it was shown that Arnold's strange duality can be generalized to a duality between so called invertible polynomials (see Section~\ref{Sect:Inv}). For an invertible polynomial $f$ the dual polynomial is the Berglund-H\"ubsch transpose $f^T$ of $f$ (Section~\ref{Sect:Inv}). The relation between the characteristic polynomials of the monodromy transformations of the dual exceptional unimodular singularities generalizes in the following way.

For a rational function $\varphi(t)$ of the form $\prod\limits_m(1-t^m)^{s_m}$, its $k$th ``power'' $\varphi^{(k)}(t)$ is defined by
$$
\varphi^{(k)}(t)=\prod_m(1-t^{\frac{m}{\gcd{(m,k)}}})^{\gcd{(m,k)}s_m}.
$$
If $\varphi(t)$ is the zeta function of a transformation $h:Z\to Z$, then $\varphi^{(k)}(t)$ is the zeta function of the $k$th power $h^k$ of the transformation $h$.

For a rational function $\varphi(t)=\prod\limits_m(1-t^m)^{s_m}$, its (formal) $k$th ``root'' (or ``root'' of degree $k$) is a rational function $\psi(t)=\varphi^{(1/k)}(t)$ of the same form such that $\psi^{(k)}(t)=\varphi(t)$. A $k$th ``root'' of a rational function does not always exist and, in general, is not well defined. Note that the function $\varphi(t)=\prod\limits_m(1-t^m)^{s_m}$ has a ``root'' of degree $k$ if and only if each factor $(1-t^m)^{s_m}$ has a ``root'' of this degree.

An invertible polynomial has a canonical system of weights which may be non-reduced (Section~\ref{Sect:Inv}).
In \cite{ET} it was shown that, if $n=3$ and $f$ is a non-degenerate invertible polynomial with the canonical system of weights being reduced, then
\begin{equation}\label{ET}
P_{\ww}(t){\rm Or}_{\ww}(t)=\widetilde{\zeta}^{(1/c^T)}_{f^T}(t),
\end{equation}
where $c^T$ is the greatest common divisor of the canonical weights of $f^T$.
Together with (\ref{ourSaito}) this implies that in this case
\begin{equation}\label{dual_reduced}
\widetilde{\zeta}^\ast_f(t)=\widetilde{\zeta}^{(1/c^T)}_{f^T}(t).
\end{equation}

We shall show that the equation (\ref{dual_reduced}) has a generalization of the form
\begin{equation}\label{duality}
\widetilde{\zeta}^{(1/c)\ast}_f(t)=\left(\widetilde{\zeta}^{(1/c^T)}_{f^T}(t)\right)^{(-1)^{n-1}}
\end{equation}
for a number of cases when the canonical systems of weights both of $f$ and $f^T$ may be non-reduced and also when $n$ is arbitrary. Moreover, we show that in these cases the ``roots'' of the zeta functions are zeta functions of roots of the monodromy transformations defined by a geometric construction.

%%%%%%%%%%%%%%%%%%%%%%
\section{Invertible polynomials} \label{Sect:Inv}
A quasihomogeneous polynomial $f$ in $n$ variables is said to be invertible (see, e.g., \cite{Kr}) if it is of the form
\begin{equation}\label{inv}
f(x_1, \ldots, x_n)=\sum\limits_{i=1}^n a_i \prod\limits_{j=1}^n x_j^{E_{ij}}
\end{equation}
for some coefficients $a_i\in\CC^\ast$ and for a matrix $E=(E_{ij})$ with positive integer entries and with $\det E\ne 0$.
For simplicity we assume that $a_i=1$ for $i=1, \ldots, n$. (This can be achieved by a rescaling of the variables $x_j$.) 
An invertible quasihomogeneous polynomial $f$ is non-degenerate if it has (at most) an isolated critical point at the origin in $\CC^n$.

According to \cite{KS}, an invertible polynomial $f$ is non degenerate if and only if it is a (Thom-Sebastiani) sum of invertible polynomials (in groups of different variables) of the following types:
%% \newline 1) $x_1^{p_1}$ (Fermat type);
\begin{enumerate}
\item[1)] $x_1^{p_1}x_2 + x_2^{p_2}x_3 + \ldots + x_{m-1}^{p_{m-1}}x_m + x_m^{p_m}$ (chain type; $m\ge 1$);
\item[2)] $x_1^{p_1}x_2 + x_2^{p_2}x_3 + \ldots + x_{m-1}^{p_{m-1}}x_m + x_m^{p_m}x_1$ (loop type; $m\ge 2$).
\end{enumerate}

\begin{sloppypar}

An invertible polynomial (\ref{inv}) has a canonical system of weights $\ww=(w_1, \ldots, w_n;d)$, where $w_i$ is the determinant of the matrix $E$ with the $i$th column substituted by $(1, \ldots, 1)^T$ and $d=\det E$. One has $f(\lambda^{w_1}x_1, \ldots, \lambda^{w_n}x_n) = \lambda^d f(x_1, \ldots, x_n)$. The canonical system of weights may be non-reduced, i.e., one may have $c=\gcd(w_1, \ldots, w_n)\ne 1$. 

\end{sloppypar}

\begin{example}
Let $f(x_1,x_2,x_3)=x_1^3x_2+x_2^4x_3+x_3^5$. One has $\ww=(16,12,12;60)$, $c=4$.
\end{example}

The Berglund-H\"ubsch transpose $f^T$ of the invertible polynomial (\ref{inv}) is defined by
$$
f^T(x_1, \ldots, x_n)=\sum\limits_{i=1}^n a_i \prod\limits_{j=1}^n x_j^{E_{ji}}\,.
$$
If the invertible polynomial $f$ is non-degenerate, then $f^T$ is non-degenerate as well.

If the canonical system of weights of $f$ is reduced, the canonical system of weights of $f^T$ can be non-reduced. The canonical (quasi)degrees of $f$ and $f^T$ coincide.

\begin{example} Let $f(x_1,x_2,x_3)=x_1^5x_2+x_2^2+x_3^3$. Then $f^T(x_1,x_2,x_3)=x_1^5+x_1x_2^2+x_3^3$. One has $\ww_f = (3,15,10;30)$, $c=1$, $\ww_{f^T}=(6,12,10;30)$, $c^T=2$.
\end{example}

%%%%%%%%%%%%%%%%
\section{Geometric roots of the monodromy}
Let $f$ be a  quasihomogeneous polynomial in $n$ variables with weights $w_1, \ldots, w_n$ and degree $d$. (We do not suppose, in general, the system of weights to be reduced.) The monodromy transformation $h_f$ of $f$ can be constructed in the following way. Let $\Gamma_t:\CC^n\to\CC^n$ be defined by $$
\Gamma_t(x_1, \ldots, x_n)=\left(\exp\left({2\pi i w_1t}/{d}\right)x_1, \ldots, \exp\left({2\pi i w_nt}/{d}\right)x_n \right).
$$
The map $\Gamma_t$ sends the Milnor fibre $V_1=f^{-1}(1)$ to the Milnor fibre $V_{\exp(2\pi it)}=f^{-1}(\exp(2\pi it))$. Then $h_f=\Gamma_1:V_1\to V_1$.

Suppose that there exists an action of the cyclic group $\ZZ_k$ on $\CC^n$ such that the generator of $\ZZ_k$ acts by 
$$
\Sigma(x_1, \ldots, x_n)=(\sigma^{m_1}x_1, \ldots, \sigma^{m_n}x_n),
$$
where $\sigma=\exp\left({2\pi i}/{k}\right)$, and
$$
f(\Sigma(x_1, \ldots, x_n))=\sigma f(x_1, \ldots, x_n).
$$
Let $\widehat{h}_f=\Sigma^{-1}\circ\Gamma_{1/k}$. The map $\widehat{h}_f$ sends the Milnor fibre $V_1$ to itself.
Since $\Gamma_t$ and $\Sigma$ commute, one has $\widehat{h}_f^k=\Sigma^{-k}\circ\Gamma_1=h_f$.
We shall call $\widehat{h}_f$ {\em a geometric  root} of degree $k$ of the monodromy transformation $h_f$.

One can see that the map $\widehat{h}_f$ is the monodromy transformation of the germ $f^k$ on the quotient $\CC^n/\ZZ_k$.

For an invertible polynomial (\ref{inv}) the required action of the group $\ZZ_k$ exists if and only if the system of equations
\begin{equation}\label{modk}
E \left( \begin{array}{c} m_1 \\ \vdots \\ m_n \end{array} \right) \equiv \left( \begin{array}{c} 1 \\ \vdots \\ 1 \end{array} \right) \mbox{ mod } k
\end{equation}
has a solution (in $\ZZ_k^n$). If $k$ is prime this means that 
$$\rk E_k=\rk\left( E_k \left| \begin{array}{c} 1 \\ \vdots \\ 1 \end{array} \right) \right. ,$$ 
where $E_k$ is the matrix obtained from $E$ by reduction modulo
$k$. If in addition $\rk E_k<n$, a solution of ($\ref{modk}$) is not unique.

%%%%%%%%%%%%%%%%%%%%%%%
\section{Duality of zeta functions}

\begin{theorem} \label{theo1}
Let $f$ be a non-degenerate invertible polynomial in $n$ variables of chain or loop type and let $c$ be the greatest common divisor of the canonical weights of $f$. Then there exist geometric roots of degree $c$ of the monodromy transformation $h_f$, all of them have the same (reduced) zeta function $\widetilde{\zeta}_f^{(1/c)}(t)$, and
\begin{equation}\label{eq_theo1}
\widetilde{\zeta}_{f^T}^{(1/c^T)}(t) =
\left(\widetilde{\zeta}_{f}^{(1/c)^\ast}(t)\right)^{(-1)^{n-1}}\, ,
\end{equation}
where $f^T$ is the Berglund-H\"ubsch transpose of $f$, $c^T$ is the greatest common divisor of the corresponding weights, and ${}^\ast$ means the Saito dual with respect to the common canonical degree $d=d^T$ of $f$ and $f^T$. 
\end{theorem}

\begin{proof}
Let 
$$
f(x_1, \ldots, x_n)=x_1^{p_1}x_2 + x_2^{p_2}x_3 + \ldots + x_{n-1}^{p_{n-1}}x_n + x_n^{p_n}
$$
be a polynomial of chain type. Let
$$
\langle q_1, \ldots, q_k\rangle := q_1\cdots q_k -  q_2\cdots q_k + \ldots + (-1)^{k-1}q_k+(-1)^k.
$$
The canonical weights of $f$ are: 
$$\langle p_2, \ldots, p_n\rangle, p_1\langle p_3, \ldots, p_n\rangle , \dots, p_1\cdots p_{n-2}\langle p_n\rangle, p_1\cdots p_{n-1};
$$
the degree is equal to $p_1\cdots p_{n}$. The equation (\ref{modk}) on the exponents defining the required $\ZZ_c$-action has the following solutions (their number being equal to $c$):
$m_1=m$, $m_2=1-mp_1$, 
$$m_j=(-1)^j\langle p_2, \ldots, p_{j-1}\rangle  + (-1)^{j-1}mp_1\cdots p_{j-1} \mbox{ for }j\ge 2,
$$
where $m$ is an arbitrary integer $\mod c$. The monodromy zeta function $\zeta_f(t)$ of $f$ can be computed, e.g., with the use of the Varchenko formula \cite{Var} and is equal to
$$
\zeta_f(t)=\prod_{j=1}^n \left(1-t^{p_j\cdots p_n/c_j}\right)^{(-1)^{n-j}c_j}\,,
$$
where $c_j=\gcd{(\langle p_{j+1} ,\ldots , p_n\rangle, p_j\langle p_{j+2} , \ldots ,p_n\rangle, \ldots, p_j\cdots p_{n-2}\langle p_n\rangle, p_j\cdots p_{n-1})}$. We shall prove that a ``root'' of degree $c$ of $\zeta_f(t)$ is equal to 
\begin{equation}\label{root}
\prod_{j=1}^n \left(1-t^{p_j\cdots p_n}\right)^{(-1)^{n-j}}\,.
\end{equation}
It looks rather envolved to check the number theoretical conditions on the exponents $p_j$ and the greatest common divisors $c_j$ and $c$ which give this fact. Moreover the formal ``root'' of $\zeta_f(t)$ may not be unique. Instead of that we shall show that the zeta function of the corresponding geometric root $\widehat{h}_f = \Sigma^{-1} \circ \Gamma_{1/c}$ is equal to (\ref{root}). This follows from the fact that, on the coordinate torus $\{(x_1, \ldots , x_n) \, : \, x_i=0 \mbox{ for }i <  j, x_i \neq 0 \mbox{ for } i \geq j\}$ corresponding to the variables $x_j, \ldots , x_n$, the order of the map $\Sigma^{-1} \circ \Gamma_{{1}/{c}}$ is equal to $p_j \cdots p_n$ (at each point), $j=1, \ldots , n$. 

One has
$$
\Sigma^{-1} \circ \Gamma_{{1}/{c}}(x_1, \ldots , x_n)= (\exp(2\pi i b_1)x_1, \ldots , \exp(2 \pi i b_n)x_n),
$$
where
$$b_j := \frac{(-1)^{j+1} \langle p_2, \ldots , p_n \rangle + (-1)^j mp_1 \cdots p_n}{c\,p_j \cdots p_n}, \quad j=1, \ldots , n.
$$
The numerator is obviously divisible by $c$ (because $w_1=\langle p_2, \ldots , p_n \rangle$ and $d=p_1 \cdots p_n$ are divisible by $c$). One has to show that the numerator divided by $c$ has no common factor with $p_j \cdots p_n$. One can see that a common factor of $p_s$ ($s=j, \ldots, n$) and the numerator is a common factor of $p_s$ and $\langle p_{s+1}, \ldots , p_n \rangle$. Therefore it is a common factor of all the weights $w_1, \ldots , w_n$ and thus divides $c$. Therefore, the numerator divided by $c$ has no common factor with $p_s$.

The polynomial $f^T$ is also given by the formula (\ref{root}) with $p_1, \ldots , p_n$ substituted by $p_n, \ldots , p_1$ respectively. One has
\begin{eqnarray*}
\widetilde{\zeta}_f^{(1/c)}(t) & = & (1-t)^{-1} \prod_{j=2}^n (1-t^{p_j \cdots p_n})^{(-1)^{n-j}} (1-t^d)^{(-1)^{n-1}}, \\
\widetilde{\zeta}_{f^T}^{(1/c^T)}(t) & = & (1-t)^{-1} \prod_{j=2}^n (1-t^{p_1 \cdots p_{j-1}})^{(-1)^j} (1-t^d)^{(-1)^{n-1}} 
\end{eqnarray*}
and therefore
$$
\widetilde{\zeta}_f^{(1/c)\ast}(t) =  (1-t^d)  \prod_{j=2}^n (1-t^{p_1 \cdots p_{j-1}})^{(-1)^{n-j-1}}(1-t)^{(-1)^n} = \left( \widetilde{\zeta}_f^{(1/c^T)} (t) \right)^{(-1)^{n-1}}.
$$

Let 
$$
f(x_1, \ldots, x_n)=x_1^{p_1}x_2 + x_2^{p_2}x_3 + \ldots + x_{m-1}^{p_{m-1}}x_m + x_m^{p_m}x_1$$
be a polynomial of loop type. The canonical weights of $f$ are
$$
\langle p_2, \ldots , p_n \rangle, \langle p_3, \ldots , p_n,p_1 \rangle, \ldots , \langle p_n, p_1, \ldots , p_{n-2} \rangle, \langle p_1, \ldots , p_{n-1} \rangle
$$
and the degree is equal to $d=p_1 \cdots p_n + (-1)^{n-1}$. The solutions of (\ref{modk}) are given by the same formulae as above. 

At first glance this looks somewhat strange: the systems of equations are different. However, one should keep in mind that they are solved modulo different greatest common divisors of the weights. Moreover, for the same exponents $p_1, \ldots , p_n$, the greatest common divisors of the weights for the chain and for the loop types are relatively prime.

The zeta function is
$$
\zeta_f(t)= (1- t^{d/c})^{(-1)^{n-1}c}.
$$
It is obvious that a formal ''root'' of degree $c$ of $\zeta_f(t)$ is equal to $(1-t^d)^{(-1)^{n-1}}$. Let us show that the zeta function of the geometric root $\widehat{h}=\Sigma^{-1} \circ \Gamma_{1/c}$ of degree $c$ of the monodromy transformation of $f$ is also equal to $(1-t^d)$. For that one has to show that the order of the map $\widehat{h}$ on the (maximal) torus $\{ (x_1, \ldots , x_n) \, : \, x_i \neq 0, i=1, \ldots , n \}$ 
is equal to $d$ at each point. One has
$$
\Sigma^{-1}\circ\Gamma_{1/c}(x_1, \ldots, x_n)=(\exp{(2\pi i b_1)}x_1, \ldots, \exp{(2\pi i b_n)}x_n)\,,
$$
where $b_j=\frac{w_j-m_jd}{cd}$. The greatest common divisor of all $w_j-m_jd$ and $d$ is just $c$ and therefore all 
ratios $\frac{w_j-m_jd}{c}$ ($j=1, \ldots, n$) together have no common factor with $d$. One has $\widetilde{\zeta}_{f}^{(1/c)}(t)=\widetilde{\zeta}_{f^T}^{(1/c^T)}(t)= (1-t^d)^{(-1)^{n-1}}/(1-t)$. This implies (\ref{eq_theo1}).
\end{proof}

\begin{example} Let $f(x_1,x_2,x_3)=x_1^3x_2+x_2^4x_3+x_3^5$, $f^T(x_1,x_2,x_3)=x_1^3+x_1x_2^4+x_2x_3^5$. One has $c=4$, $c^T=10$, $d=d^T=16$,
$$
\widetilde{\zeta}_f(t) = \frac{(1-t^5)(1-t^{60/4})^4}{(1-t^{20/4})^4(1-t)}, \quad 
\widetilde{\zeta}_{f^T}(t) = \frac{(1-t^3)(1-t^{60/10})^{10}}{(1-t^{12/2})^2(1-t)}.
$$
The reduced zeta functions of the geometric roots of the monodromy transformations of $f$ and $f^T$ are
$$
\widetilde{\zeta}^{(1/4)}_f(t) = \frac{(1-t^5)(1-t^{60})}{(1-t^{20})(1-t)}, \quad 
\widetilde{\zeta}^{(1/10)}_{f^T}(t) = \frac{(1-t^3)(1-t^{60})}{(1-t^{12})(1-t)}.
$$
They obviously are Saito dual. Note that a ''root'' of degree 4 of $\widetilde{\zeta}_f(t)$ is not well defined, e.g. 
$$
\frac{1-t^{60}}{(1-t^5)^3(1-t)}
$$
is another ''root''.
\end{example}

Let $f(x_1, x_2, x_3)=x_1^{2}+ x_2^{2} + x_3^{p}$. In this case $f^T=f$ and $f$ has a singularity of type $A_{p-1}$ at the origin. The reduced monodromy zeta function of $f$ is equal to $(1-t^p)/(1-t)$ and coincides with the reduced monodromy zeta function of the polynomial $x^p$ in one variable.
It is Saito self dual with respect to the degree $p$ of the polynomial $x^p$, but not with respect to the canonical degree of $f$ equal to $4p$. There exist geometric roots of the corresponding degrees for the polynomial $x^p$, but not for $f$.

\begin{theorem}\label{theo2}
Let $f$ be a non-degenerate invertible polynomial in 3 variables such that both $f$ and $f^T$ have critical points at the origin and f is not of the form $x_1^{2}+ x_2^2 + x_3^{p}$ (up to permutation of the variables).
Then the following statements hold:
\begin{itemize}
\item[{\rm 1)}]  If both reduced monodromy zeta functions $\widetilde{\zeta}_f(t)$ and $\widetilde{\zeta}_{f^T}(t)$ have ``roots'' (of degrees $c$ and $c^T$ respectively), then they have ``roots'' $\widetilde{\zeta}_f^{(1/c)}(t)$ and $\widetilde{\zeta}_{f^T}^{(1/c^T)}(t)$ such that
\begin{equation}\label{eq_theo2}
\widetilde{\zeta}_{f^T}^{(1/c^T)}(t) =
\widetilde{\zeta}_{f}^{(1/c)*}(t)\,,
\end{equation}
where $*$ means Saito dual with respect to the common canonical degree $d=d^T$.
\item[{\rm 2)}]   In all these cases except $f(x_1,x_2,x_3)=x_1^{p_1}x_2+x_2^{p_2}x_1+x_3^{p_3}$ with $p_3=p_1p_2-1$ and $\gcd{(p_1-1, p_2-1)}=1$, the monodromy transformations of both $f$ and $f^T$ have geometric roots, all geometric roots of each of them have the same reduced monodromy zeta functions and these functions are Saito dual to each other (with respect to $d=d^T$).
\item[{\rm 3)}]   If $\widetilde{\zeta}_f(t)$ has a ``root'' of degree $c$, then either $\widetilde{\zeta}_{f^T}(t)$ also has a ``root'' of degree $c^T$, or $f(x_1, x_2, x_3)=x_1^{2}x_2+ x_2^p + x_3^{p}$ with $p$ odd (and then $\widetilde{\zeta}_{f^T}(t)$ has no ``root'' of degree $c^T$).
\end{itemize}
\end{theorem}

\begin{proof}
This is already proved for polynomials of chain and loop types (Theorem~\ref{theo1}). (In these cases geometric roots of the monodromy transformations always exist.) We have to prove this for the other types.

In what follows we consider all polynomials up to permutation of the variables. Let $s_m$ be the power of the binomial $(1-t^m)$ in the decomposition of $\widetilde{\zeta}_f(t)$:
$$
\widetilde{\zeta}_f(t)=\prod_{m\ge 1}(1-t^m)^{s_m}.
$$

\smallskip
a) Let
$f(x_1, x_2, x_3)=x_1^{p_1}+ x_2^{p_2} + x_3^{p_3}$
(Brieskorn-Pham type). One has $f^T=f$. We assume that $p_1\le p_2\le p_3$. Note that $p_1 \geq 2$. The canonical weight system of $f$ is
$(p_2p_3, p_1p_3, p_1p_2; p_1p_2p_3)$. The reduced monodromy zeta function of $f$ is
\begin{equation}\label{zeta_a}
\widetilde{\zeta}_f(t)=\frac{\prod\limits_{i}(1-t^{p_i})(1-t^{p_1p_2p_3/c})^c}{\prod\limits_{i<j}(1-t^{p_ip_j/c_{ij}})^{c_{ij}}(1-t)}\,,
\end{equation}
where $c_{ij}=\gcd{(p_i,p_j)}$.

If the exponents $p_1$, $p_2$, $p_3$, are pairwise coprime, i.e. $c=c_{ij} =1$, then the reduced zeta-function (\ref{zeta_a}) is Saito self-dual with respect to $d=p_1p_2p_3$. The monodromy transformation of $f$ has geometric roots if and only if the exponents $p_1$, $p_2$, $p_3$ are pairwise coprime.
%% If 2 of the exponents are not coprime (i.e. $c_{ij}\ne 1$)
%% then $c_{ij}\vert c$ and the monodromy tansformation of $f$ has no geometric roots.

If $p_1<p_2$, then $s_{p_1}=1$ and in order to have a ``root'' of $\widetilde{\zeta}_f(t)$ of degree $c$ one needs $\gcd{(p_1,c)}=1$. In this case $\gcd{(p_1, p_2p_3)}=c_{12}=c_{13}=1$, $c=c_{23}=\gcd{(p_2, p_3)}$. If in addition $p_2<p_3$, then $s_{p_2}=1$ and in order to have a ``root'' of $\widetilde{\zeta}_f(t)$ one needs $\gcd{(p_2,c)}=1$.
This implies that $c=1$.

If $p_1<p_2=p_3$ and $\gcd{(p_1,c)}=1$, then $c_{23}=c=p_2$ and $s_{p_2}=2-c$. If $c\ne 2$, then $\widetilde{\zeta}_f(t)$ has no ``root'' of degree $c=p_2$. If $c=2$, then $p_1=1$ and $f$ has no critical point at the origin.

If $p_1=p_2<p_3$, then $s_{p_1}=2-p_1$ and $p_1\vert c$. If $p_1\ne 2$, then $\widetilde{\zeta}_f(t)$ has no ``root'' of degree $c$. If $p_1=2$, then $f$ is of the form $x_1^2+x_2^2+x_3^p$ excluded from
%% the ???
consideration.

Finally if $p_1=p_2=p_3$, then $c=p_1^2$, $s_{p_1}=p_1^2-3p_1+3$, $c>s_{p_1}>0$, and $\widetilde{\zeta}_f(t)$ has no ``root'' of degree $c$.

\smallskip
b) Let
$f(x_1, x_2, x_3)=x_1^{p_1}x_2+ x_2^{p_2}x_1 + x_3^{p_3}$. One has $f^T=f$. The canonical weight system of $f$ is
$(p_3(p_2-1), p_3(p_1-1), p_1p_2-1; p_3(p_1p_2-1))$. The reduced monodromy zeta function of $f$ is
\begin{equation}\label{zeta_b}
\widetilde{\zeta}_f(t)=\frac{(1-t^{p_3})(1-t^{p_3(p_1p_2-1)/c})^c}{(1-t^{(p_1p_2-1)/c_{1}})^{c_{1}}(1-t)}\,,
\end{equation}
where $c_{1}=\gcd{(p_2-1,p_1-1)}$.

If $\gcd{(p_3,c)}=1$, then $c=\gcd{(p_2-1,p_1-1)}=c_1$ and geometric roots of the monodromy transformation of $f$ of degree $c$ exist. (They exist for $x_1^{p_1}x_2+ x_2^{p_2}x_1$ due to Theorem~\ref{theo1} because $c=c_1$, and for $x_3^{p_3}$ since $\gcd{(p_3,c)}=1$.) The geometric roots have one and the same reduced zeta function
$$
\widetilde{\zeta}^{(1/c)}_f(t)=\frac{(1-t^{p_3})(1-t^{p_3(p_1p_2-1)})}{(1-t^{p_1p_2-1})(1-t)}\,,
$$
which is Saito self dual with respect to $d=p_3(p_1p_2-1)$.

Suppose that $\gcd{(p_3,c)}\ne 1$. In this case geometric roots of the monodromy transformation of $f$ do not exist.

If $p_3 \neq (p_1p_2-1)/c_1$, then $s_{p_3} \equiv 1\mod c$ and $\widetilde{\zeta}_f(t)$ has no ``root'' of degree $c$.

If $p_3=(p_1p_2-1)/c_1$, then $p_3\vert c$ and $s_{p_3} \equiv 1-c_1\mod c$. If $c_1\ne 1$, then (since $p_3\vert c$) in order to have a ``root'' of degree $c$ of $(1-t^{p_3})^{s_{p_3}}$ it is necessary that $p_3\vert (c_1-1)$, however $p_3> p_1-1$ and $p_3> p_2-1$ and therefore $p_3>c_1$. Thus $\widetilde{\zeta}_f(t)$ has no ``root'' of degree $c$. If $c_1=1$ (i.e. $p_3=p_1p_2-1$) then
$$
\widetilde{\zeta}_f(t) = \frac{(1-t^{p_3})^{p_3}}{1-t} ,
$$
 a ``root'' of $\widetilde{\zeta}_f(t)$ of degree $c=p_3$ exists, moreover it is unique, namely equal to $\widetilde{\zeta}^{(1/c)}_f(t)=\frac{1-t^{d}}{1-t}$, and thus is Saito self dual.

\smallskip
c) Let
$f(x_1, x_2, x_3)=x_1^{p_1}x_2+ x_2^{p_2} + x_3^{p_3}$. The transpose $f^T$ is obtained from $f$ by interchanging the exponents $p_1$ and $p_2$. The canonical weight system of $f$ is
$(p_3(p_2-1), p_3p_1, p_1p_2; p_1p_2p_3)$. The reduced monodromy zeta function of $f$ is
\begin{equation}\label{zeta_c}
\widetilde{\zeta}_f(t)=\frac{(1-t^{p_2})(1-t^{p_3})(1-t^{p_1p_2p_3/c})^c}{(1-t^{p_1p_2/c_1})^{c_1}
(1-t^{p_2p_3/c_2})^{c_2}(1-t)}\,,
\end{equation}
where $c_{1}=\gcd{(p_2-1,p_1)}$, $c_2=\gcd{(p_2,p_3)}$.
Geometric roots of the monodromy transformation of $f$ exist if and only if $\gcd{(p_3,c)}= 1$. (In this case $c=c_1$, geometric roots for $x_1^{p_1}x_2+ x_2^{p_2}$ exist due to Theorem~\ref{theo1} because $c=c_1$, and for $x_3^{p_3}$ since $\gcd{(p_3,c)}=1$.) In this case $c_2=1$ and the reduced zeta function of any geometric root of the monodromy transformation of $f$ is equal to 
$$
\widetilde{\zeta}^{(1/c)}_f(t)=\frac{(1-t^{p_2})(1-t^{p_3})(1-t^{p_1p_2p_3})}{(1-t^{p_1p_2})
(1-t^{p_2p_3})(1-t)}\,.
$$
In this case one also has $\gcd{(p_3,c^T)}= 1$ and
$$
\widetilde{\zeta}^{(1/c^T)}_f(t)=\frac{(1-t^{p_1})(1-t^{p_3})(1-t^{p_1p_2p_3})}{(1-t^{p_1p_2})
(1-t^{p_1p_3})(1-t)}=\widetilde{\zeta}^{(1/c)*}_f(t)\,.
$$

Suppose that $\gcd{(p_3,c)}\ne 1$. If $s_{p_3}=1$, then a ``root'' of degree $c$ of $\widetilde{\zeta}_f(t)$ does not exist. Otherwise $p_3$ has to coincide with $p_2$, or with $p_1p_2/c_1$, or with $p_2p_3/c_2$.
In all these cases $p_2\vert p_3$ (and therefore $p_3=p_2p_3/c_2$), and $p_2\vert c$.

If $p_2\ne p_3$, then $s_{p_2}$ is equal either to 1 (and therefore no ``root'' of $\widetilde{\zeta}_f(t)$ exists), or to $(1-c_1)$ (if $p_1p_2/c_1=p_2$). In the latter case, for the existence of a ``root'', it is necessary that $p_2\vert (c_1-1)$. However, $c_1<p_2$ and this can only happen if $c_1=1$. This implies that $p_1=1$. In this case $f^T$ has no critical point at the origin.

If $p_2=p_3$, then $c_2=p_2$, $p_2 | c$, and $s_{p_2} \mod c$ is equal either to $(2-p_2)$, or to $(2-c_1-p_2)$ (if $p_2=p_1p_2/c_1$, i.e. $p_1=c_1$). In the first case a ``root'' of $\widetilde{\zeta}_f(t)$ can only exist if $p_2=2$. Then $c_1=1$, $c=2$, $s_{2p_1}=-1$ and no ``root'' of $\widetilde{\zeta}_f(t)$ of degree 2 exists.
In the second case either $c_1=p_2-1=p_1$, $s_{p_2}=3$ and therefore a ``root'' exists if and only if $p_2=p_3=3$, $p_1=2$ (in this case the polynomial $f$ has a singularity of type $\widetilde{E}_6$), or $c_1\le (p_2-1)/2$, $-3p_2/2<s_{p_2}\le -p_2+1$ and a ``root'' of $\widetilde{\zeta}_f(t)$ can exist only if $s_{p_2}=-p_2$, i.e. $c_1=p_1=2$ and $p_2$ is odd.

If $p_1=2$ and $p_2=p_3=p$ is odd, 
$$
\widetilde{\zeta}_f(t)=\frac{(1-t^{p})^p}{1-t}
$$
and a ``root'' of degree $c=2p$ exists.
In this case
$$
\widetilde{\zeta}_{f^T}(t)=\frac{(1-t^2)(1-t^p)(1-t^{2p})^{p-2}}{1-t}
$$
and a ``root'' of degree $c^T=p$ does not exist.
\end{proof}

\begin{remarks}
{\bf 1.} For the invertible polynomials in the last table of \cite{ET}, no roots of the monodromy zeta function exist.

\noindent {\bf 2.}
One can show that for a non-degenerate invertible polynomial in two variables a ``root'' of the monodromy zeta function exists if and only if a geometric root of the monodromy transformation exists. In this case, a geometric root of the monodromy transformation of $f^T$ also exists and
$$
\widetilde{\zeta}_{f^T}^{(1/c^T)}(t) =
\left(\widetilde{\zeta}_{f}^{(1/c)^\ast}(t)\right)^{-1}\, .
$$
\end{remarks}

%%%%%%%%%%%%%%%%%%%%%%%%%%%%%%%%%%%%%%%

\bigskip
\noindent Leibniz Universit\"{a}t Hannover, Institut f\"{u}r Algebraische Geometrie,\\
Postfach 6009, D-30060 Hannover, Germany \\
E-mail: ebeling@math.uni-hannover.de\\

\medskip
\noindent Moscow State University, Faculty of Mechanics and Mathematics,\\
Moscow, GSP-1, 119991, Russia\\
E-mail: sabir@mccme.ru

\end{document}